\documentclass{amsart}
\usepackage{mathrsfs}
\usepackage{amssymb}
\usepackage{amsmath}
\usepackage{amsfonts}
\usepackage{amsfonts,enumerate}
\usepackage{pifont}
\usepackage{enumerate}
\usepackage{graphicx}
\usepackage{verbatim}
\usepackage{amsthm}

\numberwithin{equation}{section}
\newtheorem{theorem}{Theorem}[section]
\newtheorem{corollary}{Corollary}[section]
\newtheorem{definition}{Definition}[section]

\newtheorem{proposition}{Proposition}[section]
\newtheorem{remark}{Remark}[section]
\newtheorem{example}{Example}[section]

\newcommand{\M}{{\mathcal M}}
\newcommand{\N}{{\mathcal N}}

\newcommand{\8}{\infty}
\newcommand{\el}{\ell}

\newcommand{\be}{\begin{eqnarray*}}
\newcommand{\ee}{\end{eqnarray*}}
\newcommand{\beq}{\begin{equation}}
\newcommand{\eeq}{\end{equation}}
\newcommand{\beqn}{\begin{equation*}}
\newcommand{\eeqn}{\end{equation*}}
\newcommand{\bsp}{\begin{split}}
\newcommand{\esp}{\end{split}}

\numberwithin{equation}{section}

\begin{document}

\title{Noncommutative maximal inequalities associated with convex functions}


\author{Turdebek N. Bekjan}

\address{College of Mathematics and Systems Science, Xinjiang
University, Urumqi 830046, China}

\email{bek@xju.edu.cn}


\author{Zeqian Chen}

\address{Wuhan Institute of Physics and Mathematics, Chinese
Academy of Sciences, West District 30, Xiao-Hong-Shan, Wuhan 430071, China}

\email{zqchen@wipm.ac.cn}


\author{Adam Os\c{e}kowski}

\address{Department of Mathematics, Informatics and Mechanics, University of Warsaw, Banacha 2, 02-097 Warsaw, Poland}

\email{ados@mimuw.edu.pl}


\date{}
\maketitle

\markboth{T. N. Bekjan, Z. Chen, and A. Os\c{e}kowski}%
{Maximal inequalities}

\begin{abstract}
We prove several noncommutative maximal inequalities associated with convex functions, including a Doob type inequality for a convex function of maximal operators on noncommutative martingales, noncommutative Dunford-Schwartz and Stein maximal ergodic inequalities for a convex function of positive and symmetric positive contractions. The key ingredient in our proofs is a Marcinkiewicz type interpolation theorem for a convex function of maximal operators in the noncommutative setting, which we establish in this paper. These generalize the results of Junge and Xu in the $L^p$ case to the case of convex functions.
\end{abstract}


\section{Introduction}\label{intro}

Noncommutative martingale theory has received considerable progress since the seminal paper by Pisier and Xu \cite{PX1997} in 1997, thanks to interactions with several fields of mathematics such as operator spaces (e.g. \cite{ER2000, Pisier2003}) and free probability (e.g. \cite{HP2000, VDN1992}). Many classical martingale and ergodic inequalities have been successfully transferred to the noncommutative setting (cf. e.g. \cite{Bek2008, BCPY2010, CXY2013, Junge2002, JLeMX2006, JX2003, JX2007, JX2008, LMX2010, LMX2011, PR2006, Perrin2009, Rand2002, Rand2007, RX2014}). These inequalities of quantum probabilistic nature have, in return, applications to operator spaces, quantum stochastic analysis and noncommutative harmonic analysis. We refer to \cite{Junge2005, JP2008, JX2010, PS2002, Xu2006} for some illustrations of applications to operator space theory.

We continue this line of investigation. The aim of this paper is to prove several noncommutative maximal inequalities associated with convex functions. But the study of maximal inequalities is one of the most delicate and difficult parts in the noncommutative setting. Maximal martingale and ergodic inequalities in noncommutative $L_p$-spaces have been established respectively by Junge \cite{Junge2002} and Junge and Xu \cite{JX2007}, with the use of the techniques developed for operator space theory and theory of interpolation of Banach spaces. Their method can be generalized to obtain the corresponding maximal inequalities on noncommutative symmetric spaces, including noncommutative Lorentz and Orlicz spaces (see e.g. \cite{Dirk2012a}). However, their argument, relying heavily on Banach space properties such as duality and homogeneousness of norms, cannot be used directly to establish the corresponding inequalities associated with convex functions. Indeed, Xu inspired us to generalize their maximal inequalities to the case of convex functions. To this end, we need to establish a Marcinkiewicz type interpolation theorem for a convex function of maximal operators in the noncommutative setting, which we will prove in this paper based on some ideas of \cite{BC2012}. In the meanwhile, our argument provides an alternative and simpler proof for the results of Junge and Xu mentioned above. Recently, Dirksen \cite{Dirk2012b} showed that our interpolation theorem (Theorem \ref{th:Inter}) holds true yet when the upper control index is finite (see Remark \ref{re:Inter} (2) below).

The paper is organized as follows. In Section \ref{pre}, we present some preliminaries and notations on noncommutative martingales and noncommutative maximal operators. In Section \ref{inter}, we will define a convex function of maximal operators in the noncommutative setting and present several basic properties. Then, a noncommutative Marcinkiewicz type interpolation theorem for a convex function of maximal operators is proved, which is the key ingredient for the proofs of our main results. In Section \ref{PrTh}, we prove the main results of this paper, including a Doob type inequality for a convex function of maximal operators noncommutative martingales, noncommutative Dunford-Schwartz and Stein maximal ergodic inequalities for a convex function of positive and symmetric positive contractions. As a consequence, we obtain the noncommutative Burkholder-Davis-Gundy inequality associated with a convex function. Finally, in Section \ref{Ex}, the results obtained in the previous sections are extended to cover weak type maximal inequalities associated with convex functions.

In what follows, $C$ always denotes a constant, which may be different in different places. For two nonnegative (possibly infinite) quantities $X$ and $Y$ by $X \lesssim Y$ we mean that there exists a constant $C>0$ such that $X \leq C Y,$ and by $X \approx Y$ that $X \lesssim Y$ and $Y \lesssim X.$

\section{Preliminaries}\label{pre}

\subsection{Noncommutative Orlicz spaces}

We use standard notions from theory of noncommutative $L_{p}$-spaces. Our main references are \cite{PX2003} and \cite{Xu2007} (see also \cite{PX2003} for more bibliography). Let $\mathcal{N}$ be a semifinite von Neumann algebra acting on a Hilbert space $\mathbb{H}$ with a normal semifinite faithful trace $\nu.$ Let $L_{0}(\mathcal{N})$ denote the topological $*$-algebra of measurable operators with respect to $(\mathcal{N}, \nu).$ The
topology of $L_{0}(\mathcal{N})$ is determined by the convergence in measure. The trace $\nu$ can be extended to the positive cone
$L_{0}^{+}(\mathcal{N})$ of $L_{0}(\mathcal{N}):$
\be
\nu(x)= \int_{0}^{\infty}\lambda\,d\nu(E_{\lambda}(x)),
\ee
where  $x=\int_{0}^{\infty}\lambda\,dE_{\lambda}(x)$ is the spectral decomposition of $x$. Given $0<p<\infty,$ let
\be
L_{p}(\mathcal{N})=\{x\in L_{0}(\mathcal{N}):\;
 \nu(|x|^{p})^{\frac{1}{p}}<\infty\}.
\ee
We define
\be
\|x\|_{p}= \nu(|x|^{p})^{\frac{1}{p}},\quad x\in
L_{p}(\mathcal{N}).
\ee
Then $(L_{p}(\mathcal{N}),\|.\|_{p})$ is a Banach (or quasi-Banach for $p<1$) space. This is the noncommutative $L_{p}$-space associated
with $(\mathcal{N},\nu)$, denoted by $L_{p}(\mathcal{N},\nu)$ or simply by $L_{p}(\mathcal{N}).$ As usual, we set $L_{\infty}(\mathcal{N},\nu)=\mathcal{N}$ equipped with the operator norm.

For $x\in L_{0}(\N)$ we define
\be
\lambda_{s}(x)=\tau(e^{\perp}_s (|x|))\;(s>0)\; \; \text{and}\; \; \mu_t (x) = \inf \{ s>0:\;
\lambda_s (x) \le t \}\; (t >0),
\ee
where $e_s^{\perp} (|x|) = e_{(s,\infty)}(|x|)$ is the spectral projection of $|x|$ associated with the interval $(s,\infty).$ The function $s\mapsto\lambda_{s}(x)$ is called the {\it distribution function} of $x$ and $\mu_{t}(x)$ is the {\it generalized singular number} of $x.$ We will denote simply by $\lambda (x)$ and  $\mu(x)$ the functions $s \mapsto \lambda_s (x)$ and $t \mapsto \mu_t (x),$ respectively. It is easy to check that both are decreasing and continuous from the right on $(0,\infty).$ For further information we refer the reader to \cite{FK1986}.

For $0<p<\infty,$ we have the Kolmogorov inequality
\beq\label{eq:kol}
\lambda_{s}(x)\leq\frac{\|x\|_{p}^{p}}{s^{p}},\quad \forall s >0,
\eeq
for any $x\in L_{p}(\N).$ If $x,y$ in $L_0(\N)$, then
\beq\label{eq:tri}
\lambda_{2s}(x+y)\leq \lambda_s(x)+\lambda_s (y),\quad \forall s >0.
\eeq

Let $\Phi$  be an Orlicz function on $[0,\infty),$ i.e., a continuous increasing and convex function satisfying $\Phi(0)=0$ and $\lim_{t\rightarrow
\infty}\Phi(t)=\infty.$ Recall that $\Phi$ is said to satisfy the $\triangle_2$-condition if there is a constant $C$ such that $\Phi(2t)\leq C\Phi(t)$ for all $t>0.$ In this case, we write $\Phi \in \Delta_2.$ It is easy to check that $\Phi \in \triangle_2$ if and only if for any $a > 0$ there is a constant $C_a>0$ such that $\Phi(a t)\leq C_a \Phi(t)$ for all $t>0.$

For any $x \in L_0 (\N),$ by means of functional calculus applied to the spectral decomposition of $|x|,$ we have
\beq\label{eq:Phispectralintegral}
\nu (\Phi (|x|)) = \int^{\8}_0 \lambda_s (|x|) d \Phi (s) = \int^{\8}_0 \Phi (\mu_t (x)) d t,
\eeq
(see e.g. \cite{FK1986}). Recall that for any $x,y \in L_0 (\N)$ there exist two partial isometries $u,v\in \mathcal{N}$ such that
\beq\label{eq:OperatorModuleTriangleInequa}
| x + y| \le u^* |x| u + v^* |y| v,
\eeq
(cf. \cite{AAP1982}). Then, we have
\be
\nu (\Phi (|\alpha x + (1- \alpha) y|)) \le \alpha \nu (\Phi (|x|)) + (1-\alpha) \nu (\Phi (|y|))
\ee
for any $0 \le \alpha \le 1$ and $x , y \in L_0 (\N).$ In addition, if $\Phi \in \triangle_2,$ then
\be
\nu (\Phi (| x + y|)) \le C_{\Phi} \big [ \nu (\Phi (|x|)) + \nu (\Phi (|y|)) \big ].
\ee
We will frequently use these two inequalities in what follows.

We will work with some standard indices associated to an Orlicz function. Given an Orlicz function $\Phi,$ let
\be M(t, \Phi)= \sup_{s >0} \frac{\Phi (t s)}{\Phi (s)},\quad t >0.
\ee
Define
\be
p_{\Phi} = \lim_{t \searrow 0} \frac{\log M(t, \Phi)}{\log t}, \quad q_{\Phi} = \lim_{t \nearrow \8} \frac{\log M(t, \Phi)}{\log t}.
\ee
Note the following properties:
\begin{enumerate}[{\rm (1)}]

\item $1 \le p_{\Phi} \le q_{\Phi} \le \8.$

\item The following characterizations of $p_{\Phi}$ and $q_{\Phi}$ hold
\be p_{\Phi} = \sup \Big \{ p >0:\; \int^t_0 s^{-p} \Phi (s) \frac{d s}{s} = O(t^{- p} \Phi (t)),\; \forall t >0 \Big \};\ee
\be q_{\Phi} = \inf \Big \{ q >0:\; \int^{\8}_t s^{-q} \Phi (s) \frac{d s}{s} = O(t^{- q} \Phi (t)),\; \forall t >0 \Big \}.\ee

\item $\Phi \in \triangle_2$ if and only if $q_{\Phi} < \8,$ or equivalently, $ \sup_{t>0} t \Phi'(t)/\Phi(t)< \8.$ ($\Phi' (t)$ is defined for each $t > 0$ except for a countable set of points in which we take $\Phi'(t)$ as the derivative from the right.)

\end{enumerate}
See \cite{M1985, M1989} for more information on Orlicz functions and Orlicz spaces.

For an Orlicz function $\Phi,$ the noncommutative Orlicz space $L_{\Phi}(\mathcal{N})$ is defined as the space of all measurable operators $x$ with
respect to $(\mathcal{N},\nu)$ such that
\be
\nu \Big ( \Phi \Big ( \frac{|x|}{c} \Big ) \Big )<\infty
\ee
for some $c>0.$ The space $L_{\Phi}(\mathcal{N}),$ equipped with the norm
\be
\|x\|_{\Phi}= \inf \big \{c>0: \;\nu \big ( \Phi({|x|}/{c}) \big )<1 \big \},
\ee
is a Banach space. If $\Phi(t)=t^p$ with $1 \leq p<\infty$ then $L_\Phi(\mathcal{N})= L_p(\mathcal{N}).$ Note that if $\Phi \in \triangle_2,$ then for $x \in L_0 (\N),$ $\nu (\Phi (|x|)) < \8$ if and only if $x \in L_{\Phi}({\mathcal{N}}).$ Noncommutative Orlicz spaces are symmetric spaces of measurable operators as defined in \cite{DDP1989, Xu1991}.

Let $a=(a_{n})$ be a finite sequence in $L_{\Phi}({\mathcal{N}}).$ We define
\be
\|a\|_{L_{\Phi}({\mathcal{N}},\el_{C}^{2})}= \Big \| \Big ( \sum_n
|a_{n}|^{2} \Big )^{\frac{1}{2}} \Big \|_{\Phi}\; \text{and}\;
\|a\|_{L_{\Phi}({\mathcal{N}},\el_{R}^{2})} = \Big \| \Big ( \sum_n |a_{n}^{*}|^{2} \Big )^{\frac{1}{2}} \Big \|_{\Phi},
\ee
respectively. This gives two norms on the family of all finite sequences in $L_{\Phi}({\mathcal{N}})$ (see \cite{BC2012} for details).
The corresponding completion $L_{\Phi}({\mathcal{N}},\el_{C}^{2})$ is a Banach space. Since $\Phi$ is a continuous increasing function, by \eqref{eq:Phispectralintegral} we have that a sequence $a=(a_{n})_{n\geq 0}$ in $L_{\Phi}({\mathcal{N}})$ belongs to $L_{\Phi}({\mathcal{N}}, \el_{C}^{2})$ if and only if
\be
\sup_{n\geq 0} \Big \| \Big ( \sum_{k= 0 }^{n}|a_{k}|^{2} \Big )^{\frac{1}{2}} \Big \|_{\Phi}<\infty.
\ee
If this is the case, $\big ( \sum_{k= 0 }^{\infty}|a_{k}|^{2} \big )^{\frac{1}{2}}$ can be appropriately defined as an element of $L_{\Phi}({\mathcal{N}}).$ Similarly, $\| \cdot \|_{L_{\Phi}({\mathcal{N}},\el_{R}^{2})}$ is also a norm on the family
of all finite sequence in $L_{\Phi}({\mathcal{N}}),$ and the corresponding completion $L_{\Phi}({\mathcal{N}}, \el_{R}^{2})$ is a Banach space, which is isometric to the row subspace of $L_{\Phi}({\mathcal{N}}\otimes {\mathcal{B}}(\el^{2}))$ consisting of matrices whose nonzero entries lie only in the first row. Observe
that the column and row subspaces of $L_{\Phi}({\mathcal{N}}\otimes {\mathcal{B}}(\el^{2}))$ are 1-complemented by Theorem 3.4 in \cite{DDP1992}.

In what follows, unless otherwise specified, we always denote by $\Phi$ an Orlicz function.

\subsection{Noncommutative martingales}

Let ${\mathcal{M}}$ be a finite von Neumann algebra with a normalized normal faithful trace $\tau.$ Let
$({\mathcal{M}}_{n})_{n\geq 0}$ be an increasing sequence of von Neumann subalgebras of ${\mathcal{M}}$ such that $\cup_{n\geq 0}
{\mathcal{M}}_{n}$ generates ${\mathcal{M}}$ (in the $w^{*}$-topology). $({\mathcal{M}}_{n})_{n\geq 0}$ is called a
filtration of ${\mathcal{M}}.$ The restriction of $\tau$ to ${\mathcal{M}}_{n}$ is still denoted by $\tau.$ Let
${\mathcal{E}}_{n}={\mathcal{E}}(\cdot|{\mathcal{M}}_{n})$ be the conditional expectation of ${\mathcal{M}}$ with respect to
${\mathcal{M}}_{n}.$ Then ${\mathcal{E}}_{n}$ is a norm 1 projection of $L_{\Phi}({\mathcal{M}})$ onto $L_{\Phi}({\mathcal{M}}_{n})$ (cf. \cite[Theorem 3.4]{DDP1992}) and ${\mathcal{E}}_{n}(x)\geq 0$ whenever $x\geq 0.$

A noncommutative $L_{\Phi}$-martingale with respect to $({\mathcal{M}}_{n})_{n\geq 0}$ is a sequence $x=(x_{n})_{n\geq 0}$
such that $x_{n} \in L_{\Phi}({\mathcal{M}}_{n})$ and
\be
{\mathcal{E}}_n(x_{n+1})=x_n
\ee
for any $n \ge 0.$ Let $\|x\|_{\Phi}=\sup_{n\geq 0}\|x_{n}\|_{\Phi}.$ If $\|x\|_{\Phi} <\infty,$ then $x$ is said to be a bounded $L_{\Phi}$-martingale.

\begin{remark}\label{re:SemifiniteMart}\rm

Let ${\mathcal{M}}$ be a semifinite von Neumann algebra with a semifinite normal faithful trace $\tau$. Let
$({\mathcal{M}}_{n})_{n\geq 0}$ be  a filtration of ${\mathcal{M}}$ such that the restriction of $\tau$ to each ${\mathcal{M}}_{n}$ is
still semifinite. Then we can define noncommutative martingales with respect to $({\mathcal{M}}_{n})_{n\geq 0}$. All results on
noncommutative martingales that will be presented below can be extended to this semifinite setting.

\end{remark}

Let $x$ be a noncommutative martingale. The martingale difference sequence of $x,$ denoted by $dx=(dx_{n})_{n\geq 0},$ is defined as
\be
dx_{0}=x_{0},\quad dx_{n}=x_{n}-x_{n-1},\quad n\geq 1.
\ee
Set
\be
S^C_n (x)= \Big ( \sum_{k= 0  }^{n}|dx_{k}|^{2} \Big )^{\frac{1}{2}} \quad \mbox{and}
\quad S^R_n (x)= \Big ( \sum_{k=0}^{n}|dx_{k}^{*}|^{2} \Big )^{\frac{1}{2}}.
\ee
By the preceding discussion, $dx$ belongs to $L_{\Phi}({\mathcal{M}},\el_{C}^{2})$ (resp. $L_{\Phi}({\mathcal{M}}, \el_{R}^{2}))$ if and only if $(S^C_n (x))_{n \geq 0}$
(resp. $(S^R_n (x))_{n\geq 0}$) is a bounded sequence in $L_{\Phi}({\mathcal{M}});$ in this case,
\be
S^C (x)= \Big ( \sum_{k= 0}^{\infty} | d x_{k} |^{2} \Big )^{\frac{1}{2}} \quad \mbox{and} \quad S^R (x)= \Big ( \sum_{k= 0 }^{\infty} | d x_{k}^* |^{2} \Big )^{\frac{1}{2}}
\ee
are elements in $L_{\Phi}({\mathcal{M}}).$ These are noncommutative analogues of the usual square functions in the commutative martingale theory. It should be pointed out that the two sequences $S^C_n (x)\: \mbox{and}\: S^R_n (x)$ may not be bounded in
$L_{\Phi}({\mathcal{M}})$ at the same time.

We define $\mathcal{H}^C_{\Phi}({\mathcal{M}})$ (resp. $\mathcal{H}^R_{\Phi}({\mathcal{M}})$) to be the space of all $L_{\Phi}$-martingales such that $dx \in L_{\Phi}({\mathcal{M}}, \el_{C}^{2})$ (resp. $dx \in L_{\Phi}({\mathcal{M}}, \el_{R}^{2})$ ), equipped with the norm
\be
\|x\|_{\mathcal{H}^C_{\Phi}({\mathcal{M}})}=\|dx\|_{L_{\Phi}({\mathcal{M}}, \el_{C}^{2})
} \quad \big ( \mbox{resp.} \;
\|x\|_{\mathcal{H}^R_{\Phi}({\mathcal{M}})}=\|dx\|_{L_{\Phi}({\mathcal{M}}, \el_{R}^{2}) } \big ).
\ee
$\mathcal{H}^C_{\Phi}({\mathcal{M}})$ and $\mathcal{H}^R_{\Phi}({\mathcal{M}})$ are Banach spaces. Note that if $x \in \mathcal{H}^C_{\Phi}({\mathcal{M}}),$
\be
\|x\|_{\mathcal{H}^C_{\Phi}({\mathcal{M}})} = \sup_{n \geq 0}\|S^C_n (x)\|_{L_{\Phi}({\mathcal{M}})} = \|S^C (x)\|_{L_{\Phi}({\mathcal{M }})}.
\ee
Similar equalities hold for $\mathcal{H}^R_{\Phi}({\mathcal{M}}).$

Now, we define the Orlicz-Hardy spaces of noncommutative martingales as follows: If $q_{\Phi} < 2,$ then
\be
\mathcal{H}_{\Phi}({\mathcal{M}}) = \mathcal{H}^C_{\Phi}({\mathcal{M}}) + \mathcal{H}^R_{\Phi}({\mathcal{M}}),
\ee
equipped with the norm
\be
\|x\| = \inf \left \{ \|y\|_{ \mathcal{H}^C_{\Phi}({\mathcal{M}})} + \|z\|_{\mathcal{H}^R_{\Phi}({\mathcal{M}})}:\;
x=y+z,\; y \in \mathcal{H}^C_{\Phi}({\mathcal{M}}),\; z \in \mathcal{H}^R_{\Phi} ({\mathcal{M}}) \right \}.
\ee
If $2\leq p_{\Phi},$
\be
\mathcal{H}_{\Phi}({\mathcal{M}}) = \mathcal{H}^C_{\Phi} ({\mathcal{M}}) \cap \mathcal{H}^R_{\Phi}({\mathcal{M}}),
\ee
equipped with the norm
\be
\|x\| = \max \left \{ \|x\|_{\mathcal{H}^C_{\Phi}({\mathcal{M}})},\; \|x\|_{\mathcal{H}^R_{\Phi}({\mathcal{M}})} \right \}.
\ee

We refer to \cite{BC2012} for more information on $ \mathcal{H}_{\Phi}({\mathcal{M}}).$

\subsection{The space $L_p (\M; \el^{\8})$}

Given $1 \le p < \8,$ recall that $L_p(\M; \el^{\8})$ is defined as the space of all sequences $(x_n)_{n \ge 1}$ in $L_p (\M)$ for which there exist $a, b \in L_{2p}(\M)$ and a bounded sequence $(y_n)_{n \ge 1}$ in $\M$ such that $x_n = a y_n b$ for all $n \ge 1.$ For such a sequence, set
\beq\label{eq:p-MaxNorm}
\left \| ( x_n )_{n \ge 1} \right \|_{L_p (\M, \el^{\8})} : = \inf \big \{ \| a \|_{2 p} \sup_n\| y_n \|_{\8} \| b \|_{2 p} \big \},
\eeq
where the infimum runs over all possible factorizations of $(x_n)_{n \ge 1}$ as above. This is a norm and $L_p (\M; \el^{\8})$ is a Banach space. These spaces were first introduced by Pisier \cite{Pisier1998} in the case when $\M$ is hyperfinite and by Junge \cite{Junge2002} in the general case, and studied extensively by Junge and Xu \cite{JX2007}.

As in \cite{JX2007}, we usually write
\be
\big \| {\sup_n}^+ x_n \big \|_p = \| ( x_n )_{n \ge 1} \|_{L_p (\M, \el^{\8})}.
\ee
We warn the reader that this suggestive notation should be treated with care. It is used for
possibly nonpositive operators and
\be
\big \| {\sup_n}^+ x_n \big \|_p \neq \big \| {\sup_n}^+ | x_n | \big \|_p
\ee
in general. However it has an intuitive description in the positive case, as observed in \cite[p.329]{JX2007}: A positive sequence
$(x_n)_{n \ge1}$ of $L_p (\M )$ belongs to $L_p (\M; \el^{\8})$ if and only if there exists a positive $a \in L_p (\M)$ such
that $x_n \le a$ for any $n \ge 1$ and in this case,
\beq\label{eq:PositiveSequence}
\big \| {\sup_n}^+ x_n \big \|_p = \inf \big \{\| a \|_p :\; a \in L_p (\M),\; x_n \le a, \; \forall n \ge 1 \big \}.
\eeq
In particular, it was proved in \cite{JX2007} that the spaces $L_p (\M; \el^{\8})$ for all $1 \le p \le \8$ form interpolation scales with respect to complex interpolation. However, this result is no longer true for the real interpolation. This is one of the
difficulties one will encounter for dealing with Marcinkiewicz type interpolation theorem on maximal operators in the noncommutative setting.

\section{Interpolation}\label{inter}

In this section, we will establish a noncommutative Marcinkiewicz type interpolation theorem for a convex function of maximal operators, which plays a crucial role in the proofs of our main results in the next section.

To this end, we introduce the following definition.

\begin{definition}\label{df:QuasiOperator}
Let $1 \le p_0 < p_1 \le \8.$ Let $S = (S_n)_{n \ge 1}$ be a sequence of maps from $L^+_{p_0}(\mathcal{M}) + L^+_{p_1}(\mathcal{M}) \mapsto L^+_{0}(\mathcal{M}).$
\begin{enumerate}[\rm (1)]

\item $S$ is said to be subadditive, if for any $n \ge 1,$
\be
S_n (x + y) \le S_n (x) + S_n (y),\quad \forall x, y \in L^+_{p_0}(\mathcal{M}) + L^+_{p_1}(\mathcal{M}).
\ee

\item $S$ is said to be of weak type $(p, p)$ ($p_0 \le p < p_1$) if there is a positive constant $C$ such that for any $x \in L^+_p (\M)$ and any $\lambda >0$ there exists a projection $e \in \M$ such that
\be
\tau (e^{\perp} ) \le \left ( \frac{C \| x \|_p}{\lambda} \right )^p \quad \text{and} \quad e S_n (x) e \le \lambda,\; \forall n \ge 1.
\ee

\item $S$ is said to be of type $(p, p)$ ($p_0 \le p \le p_1$) if there is a positive constant $C$ such that for any $x \in L^+_p (\M)$ there exists $a \in L^+_p (\M)$ satisfying
\be
\| a \|_p \le C \| x \|_p \quad \text{and} \quad S_n (x) \le a,\; \forall n \ge 1.
\ee
In other words, $S$ is of type $(p, p)$ if and only if $\| S(x) \|_{L_p (\M; \el^{\8})} \le C \| x \|_p$ for all $x \in L^+_p (\M).$

\end{enumerate}
\end{definition}

This definition of subadditive operators in the noncommutative setting is due to Junge and Xu \cite{JX2007}, who proved a noncommutative analogue of the classical Marcinkiewicz interpolation theorem as follows.

\begin{theorem}\label{th:InterJX2007}{\rm (cf. \cite[Theorem 3.1]{JX2007})}\;
Let $1 \le p_0 < p_1 \le \8.$ Let $S = (S_n)_{n \ge 1}$ be a sequence of maps from $L^+_{p_0}(\mathcal{M}) + L^+_{p_1}(\mathcal{M}) \mapsto L^+_{0}(\mathcal{M}).$ Assume that $S$ is subadditive. If $S$ is of weak type $(p_0, p_0)$ with constant $C_0$ and of type $(p_1, p_1)$ with constant $C_1,$ then for any $p_0 < p < p_1,$ $S$ is of type $(p, p)$ with constant $C_p$ satisfying
\be
C_p \le C C^{1- \theta}_0 C^{\theta}_1 \left ( \frac{1}{p_0} - \frac{1}{p} \right )^{- 2}
\ee
where $\theta$ is determined by $1 /p = (1- \theta) /p_0 + \theta /p_1$ and $C$ is an absolute constant.
\end{theorem}

To state our results, we need to define a convex function of maximal operators in the noncommutative setting as follows.

\begin{definition}\label{df:PhiMaxOper}
Let $\Phi$ be an Orlicz function. Let $(x_n)$ be a sequence in $L_{\Phi} (\M).$ We define $\tau [ \Phi ( {\sup_n}^+ x_n ) ]$ by
\beq\label{eq:PhiMaxSequ}
\tau \Big [ \Phi \big ( {\sup_n}^+ x_n \big ) \Big ] : = \inf \left \{ \frac{1}{2} \Big ( \tau \big [ \Phi \big ( |a|^2 \big ) \big ] + \tau \big [ \Phi \big ( |b|^2 \big ) \big ]\Big ) \sup_n \| y_n \|_{\8} \right \}
\eeq
where the infimum is taken over all decompositions $x_n = a y_n b$ for $a, b \in L_0 (\M)$ and $(y_n) \subset L_{\8} (\M)$ with $|a|^2, |b|^2 \in L_{\Phi} (\M),$ and $\| y_n \|_{\8} \le 1$ for all $n.$
\end{definition}

To understand $\tau \big [ \Phi \big ( {\sup_n}^+ x_n \big ) \big ],$ let us consider a positive sequence $x = (x_n)$ in $L_{\Phi} (\M).$ We then note that
\beq\label{eq:PhiMaxSequPosiGe}
\tau \Big [ \Phi \big ( {\sup_n}^+ x_n \big ) \Big ] \le \tau \big [ \Phi \big ( a \big ) \big ],
\eeq
if $a \in L_{\Phi}^+ (\M)$ such that $x_n \le a$ for all $n.$ Indeed, for every $n$ there exists a contraction $u_n$ such that $x_n^{\frac{1}{2}} = u_n a^{\frac{1}{2}}$ and hence $x_n = a^{\frac{1}{2}} u^*_n u_n a^{\frac{1}{2}}.$ This concludes \eqref{eq:PhiMaxSequPosiGe}. Moreover, the converse to \eqref{eq:PhiMaxSequPosiGe} also holds true provided $\Phi \in \triangle_2.$

\begin{proposition}\label{prop:BasicPhiMax}
Let $\Phi$ be an Orlicz function satisfying the $\triangle_2$-condition.
\begin{enumerate}[\rm (1)]

\item If $x = (x_n)$ is a positive sequence in $L_{\Phi} (\M),$ then
\be
\tau \Big [ \Phi \big ( {\sup_n}^+ x_n \big ) \Big ] \approx \inf \Big \{ \tau \big [ \Phi \big ( a \big ) \big ]:\; a \in L_{\Phi}^+ (\M)\; \text{such that}\; x_n \le a, \forall n \ge 1 \Big \}.
\ee

\item For any two sequences $x =(x_n), y = (y_n)$ in $L_{\Phi} (\M)$ one has
\be
\tau \Big [ \Phi \big ( {\sup_n}^+ (x_n + y_n) \big ) \Big ] \lesssim \tau \Big [ \Phi \big ( {\sup_n}^+ x_n \big ) \Big ] + \tau \Big [ \Phi \big ( {\sup_n}^+ y_n \big ) \Big ].
\ee


\end{enumerate}
\end{proposition}

\begin{proof}
(1).\; Let $(x_n)$ be a sequence of positive elements in $L_{\Phi} (\M).$ Suppose $x_n = a y_n b$ with $|a|^2, |b|^2 \in L_{\Phi} (\M)$ and $\sup_n \| y_n\|_{\8} \le 1.$ Without loss of generality, we can assume $a, b \ge 0.$ Set $c = (a^2 + b^2 )^{1/2}.$ Then there exist two partial isometries $u, v \in \M$ such that
\be
a = c u\quad \text{and}\quad b = v c,
\ee
i.e., $x_n = c u y_n v c$ for all $n,$ and $\sup_n \| u y_n v \|_{\8} \le 1.$ Thus, $x_k \le c^2 \sup_n \| y_n \|_{\8}$ for all $k.$ By the $\triangle_2$-condition, one has
\be\begin{split}
\tau [ \Phi ( c^2 \sup_n \| y_n \|_{\8}) ] & \le \sup_n \| y_n \|_{\8} \tau [ \Phi ( c^2 ) ]\\
& \le C_{\Phi} \sup_n \| y_n \|_{\8} \frac{1}{2} \Big ( \tau \big [ \Phi \big ( |a|^2 \big ) \big ] + \tau \big [ \Phi \big ( |b|^2 \big ) \big ]\Big ).
\end{split}\ee
Combining this with \eqref{eq:PhiMaxSequPosiGe} completes the proof of (1).

(2).\; We have the following useful description of $\tau [ \Phi ( {\sup_n}^+ x_n ) ]:$
\beq\label{eq:PhiMaxSequ2}
\tau \Big [ \Phi \big ( {\sup_n}^+ x_n \big ) \Big ] = \inf \left \{ \frac{1}{2} \Big ( \tau \big [ \Phi \big ( |a|^2 \big ) \big ] + \tau \big [ \Phi \big ( |b|^2 \big ) \big ]\Big ) \right \},
\eeq
where the infimum is taken over all decompositions $x_n = a y_n b$ for $a, b \in L_0 (\M)$ and $(y_n) \subset L_{\8} (\M)$ with $|a|^2, |b|^2 \in L_{\Phi} (\M),$ and $\sup_n \| y_n \|_{\8} = 1.$ Indeed, for a decomposition $x_n = a y_n b$ with $\sup_n \| y_n \|_{\8} \le 1,$ we set $\tilde{a} = \lambda^{1/2} a,$ $\tilde{b} = \lambda^{1/2} b,$ and $\tilde{y}_n = y_n / \lambda$ with $\lambda = \sup_n \| y_n \|_{\8}.$ Then $x_n = \tilde{a} \tilde{y}_n \tilde{b}$ for all $n$ and $\sup_n \| \tilde{y}_n \|_{\8} =1,$ so that
\be
\tau \big [ \Phi \big ( |\tilde{a}|^2 \big ) \big ] + \tau \big [ \Phi \big ( |\tilde{b}|^2 \big ) \big ] \le \lambda \Big ( \tau \big [ \Phi \big ( |a|^2 \big ) \big ] + \tau \big [ \Phi \big ( |b|^2 \big ) \big ] \Big ).
\ee
This concludes \eqref{eq:PhiMaxSequ2}.

Now, to obtain the required inequality, it suffices to repeat the proof of the first part of \cite[Theorem 3.2]{DJ2004} through using \eqref{eq:PhiMaxSequ2}. We omit the details.
\end{proof}

\begin{remark}\label{rk:MaxOrliczSpace}\rm
For a sequences $x =(x_n)$ in $L_{\Phi} (\M),$ set
\be
\big \| {\sup_n}^+ x_n \big \|_{\Phi} : = \; \inf \Big \{ \lambda > 0:\; \tau \Big [ \Phi \Big ( {\sup_n}^+ \frac{x_n}{\lambda} \Big ) \Big ] \le 1 \Big \}.
\ee
One can check that $\| {\sup_n}^+ x_n \|_{\Phi}$ is a norm in $x = (x_n).$ Define
\be
L_{\Phi} (\M; \el^{\8}): = \; \Big \{(x_n) \subset L_{\Phi} (\M):\; \tau \Big [ \Phi \Big ( {\sup_n}^+ \frac{x_n}{\lambda} \Big ) \Big ] < \8\; \text{for some}\; \lambda >0 \Big \},
\ee
equipped with $\| (x_n) \|_{L_{\Phi} (\M; \el^{\8})} = \| {\sup_n}^+ x_n \|_{\Phi}.$ Then $L_{\Phi} (\M; \el^{\8})$ is a Banach space. For $1 \le p < \8,$ if $\Phi (t) = t^p$ then $L_{\Phi} (\M; \el^{\8}) = L_p (\M; \el^{\8})$ with equivalent norms. The details are left to the interested readers.
\end{remark}

We are ready to state and prove the main result of this section.

\begin{theorem}\label{th:Inter}
Let $S = (S_n)_{n \ge 0}$ be a sequence of maps from $L^+_1(\mathcal{M}) + L^+_{\8}(\mathcal{M}) \mapsto L^+_{0}(\mathcal{M}).$ Let $1 \le p < \8.$ Assume that $S$ is subadditive, and order preserving in the sense that for all $n \ge 1,$ $S_n(x) \le S_n(y)$ whenever $x \le y$ in $L^+_0 (\M).$ If $S$ is simultaneously of weak type $(p, p)$ with constant $C_p$ and of type $(\8, \8)$ with constant $C_{\8},$ then for an Orlicz function $\Phi$ with $p <p_{\Phi}\le q_{\Phi}< \8,$ there exists a positive constant $C$ depending only on $C_p, C_{\8}, p_{\Phi}$ and $q_{\Phi},$ such that
\begin{equation}\label{eq:PhiMaxOper}
\tau \Big [ \Phi \big ( {\sup_n}^+ S_n( x) \big ) \Big ]\leq C \tau \big [ \Phi( x) \big ],
\end{equation}
for all $x \in L^+_{\Phi}(\mathcal{M}).$
\end{theorem}

\begin{proof}
Since $S$ is of weak type $(p, p)$ with constant $C_p,$ for any $x \in L^+_{p} (\M)$ and each $\lambda >0$ there is a projection $q^{(\lambda)} \in \M$ such that
\be
\tau (1- q^{(\lambda)}) \le \frac{C^{p}_{p} \tau (x^p)}{\lambda^{p}} \quad \text{and} \quad q^{(\lambda)} S_n (x) q^{(\lambda)} \le \lambda q^{(\lambda)},\quad \forall n \ge 1.
\ee
For any $k \in \mathbb{Z}$ we set
\be
q_k = \bigwedge_{j \ge k} q^{(2^j)} \quad \text{and} \quad p_k = q_k - q_{k-1}.
\ee
We claim the following two facts.
\begin{enumerate}[\rm (i)]

\item  $q_k S_n(x) q_k \le 2^k q_k$ and
\beq\label{eq:weak(p,p)-inequ}
\tau (1- q_k) \le \frac{C^p_p}{1-2^{-p}} \frac{\tau (x^p)}{2^{kp}},\quad \forall k \in \mathbb{Z}.
\eeq

\item Suppose in addtion, that $x \in \M.$ Fix an integer $N$ and a sequence $(\alpha_k)^N_{k = -\8}$ of positive numbers for which $\sum_{k \le N} \frac{2^k}{\alpha_k}  < \8.$ Then the operator
\beq\label{eq:majorant}
a = 2 C_{\8} \| x\| (1- q_N) + 2 \Big ( \sum_{k \le N} \frac{2^k}{\alpha_k}\Big ) \sum_{k \le N} \alpha_k p_k.
\eeq
is a majorant of $S(x),$  i.e., $S_n (x) \le a$ for all $n \ge 1.$
\end{enumerate}

To prove these two statements, note that
\be
\tau (1- q_k) \le \sum_{j \ge k} \tau (1- q^{(2^j)}) \le C^p_p \tau (x^p) \sum_{j \ge k} 2^{-j p} = \frac{C^p_p}{1-2^{-p}} \frac{\tau (x^p)}{2^{kp}},
\ee
which proves \eqref{eq:weak(p,p)-inequ}. On the other hand, for a fixed $\xi \in \mathbb{H}$ we have
\be\begin{split}
(q_N S_n (x) q_N \xi, \xi ) & = \Big ( \sum_{k,m \le N} p_k S_n (x) p_m \xi, \xi \Big ) \le \sum_{k,m \le N} \| p_k S_n (x) p_m \| \| p_k \xi\| \|p_m \xi \|\\
& \le \sum_{k,m \le N} \| p_k S_n (x) p_k \|^{\frac{1}{2}} \| p_m S_n (x) p_m \|^{\frac{1}{2}}\| p_k \xi\| \|p_m \xi \|\\
& = \Big ( \sum_{k \le N} \| p_k S_n (x) p_k \|^{\frac{1}{2}} \| p_k \xi\| \Big )^2.
\end{split}\ee
Since $p_k S_n(x) p_k \le 2^k p_k$ and so $\| p_k S_n (x) p_k \| \le 2^k,$ one concludes that
\be
(q_N S_n (x) q_N \xi, \xi ) \le \Big ( \sum_{k \le N} \frac{2^k}{\alpha_k} \Big ) \sum_{k \le N} \alpha_k \| p_k \xi\|^2 = (a_N \xi, \xi),
\ee
where $a_N = \Big ( \sum_{k \le N} \frac{2^k}{\alpha_k}\Big ) \sum_{k \le N} \alpha_k p_k.$ Note that
\be
S_n (x) \le 2 q_N S_n (x) q_N + 2 (1-q_N) S_n (x) (1-q_N).
\ee
Thus, $a$ is a majorant of $S(x).$

Take $x \in L^+_{\Phi} (\M)$ and introduce
\be
\tilde{x} = \sum_{i \in \mathbb{Z}} 2^{i +1} E_{(2^i, 2^{i+1} ]}(x) = \sum_{i \in \mathbb{Z}} 2^i e_i,
\ee
where $e_i = E_{(2^i, \8)}(x).$ For a fixed $e_i,$ we will construct a suitable majorant of the sequence $S(e_i) = (S_n (e_i) )_{n \ge 1}.$ To this end, we take $p< \beta < p_{\Phi}$ and set
\be
\delta = \frac{1}{2} \Big ( \frac{1}{p} + \frac{1}{\beta} \Big )\quad \text{and}\quad \alpha_k = 2^{- (N-k)p \delta},
\ee
where $N$ is an integer and will be fixed later on such that it only depends on $p$ and $C_p.$ By \eqref{eq:majorant} we obtain the corresponding majorant of the sequence $S(e_i) = (S_n (e_i) )_{n \ge 1},$ denoted by $a_i.$ We claim that there exists a constant $C>0$ depending only on $C_p, C_{\8}, p$ and $\beta$ such that
\beq\label{eq:majorantSingularValueControl}
\mu_t (a_i) \le C h \Big ( \frac{t}{\tau(e_i)} \Big )\quad \text{and}\quad \tilde{\mu}_t (a_i) \le \frac{C}{1- \delta} h \Big ( \frac{t}{\tau(e_i)} \Big ), \quad \forall t > 0,
\eeq
where $h(t) = \min \{t^{-\delta},\; 1\}$ and, $\tilde{\mu}_t (x) = \frac{1}{t} \int^t_0 \mu_s (x) d s$ for any $x \in L_0 (\M)$ and all $t >0.$

Indeed, an immediate computation yields that
\be
\mu (a_i) \le C_{p, \beta, N, \8} \Big ( \chi_{[0, \tau (1-q_N))} + \sum_{k \le N} 2^{- \delta (N-k)p} \chi_{(\tau (1-q_k)), \tau (1-q_{k-1})]}\Big ),
\ee
where $C_{p, \beta, N, \8} = 2( C_\8 + \sum_{k \le N} 2^k / \alpha_k).$
By \eqref{eq:weak(p,p)-inequ} one has
\be
\mu (a_i) \le C_{p, \beta, N, \8} \Big ( \chi_{[0, C_p' 2^{-pN} \tau(e_i))} + \sum_{k \le N} 2^{- \delta (N-k)p} \chi_{(C_p' 2^{-kp} \tau (e_i), C_p' 2^{(-k+1)p}\tau (e_i)]}\Big ),
\ee
where $C_p' = C^p_p /(1-2^{-p}).$ Since for any $t \in (C_p' 2^{-kp} \tau (e_i), C_p' 2^{(-k+1)p}\tau (e_i)],$
\be
h( 2^{(N-k)p} ) \le h \Big ( 2^{-p} \frac{2^{pN} t}{C_p'\tau(e_i)} \Big ),
\ee
it follows that
\be\begin{split}
\mu_t (a_i) & \le C_{p, \beta, N, \8} \Big ( \chi_{[0, C_p' 2^{-pN} \tau(e_i))} (t) + h \Big ( 2^{-p} \frac{2^{pN} t}{C_p'\tau(e_i)} \Big ) \chi_{(C_p' 2^{-pN} \tau (e_i), \8)}(t) \Big )\\
& = C_{p, \beta, N, \8} \Big [ \chi_{[0, 1]} \Big ( \frac{2^{pN} t}{C_p'\tau(e_i)} \Big ) + h \Big ( 2^{-p} \frac{2^{pN} t}{C_p'\tau(e_i)} \Big ) \chi_{(1, \8)}  \Big ( \frac{2^{pN} t}{C_p'\tau(e_i)} \Big ) \Big ]\\
& \le C_{p, \beta, N, \8} h \Big ( 2^{-p} \frac{2^{pN} t}{C_p'\tau(e_i)} \Big ) \le C h \Big ( \frac{t}{\tau(e_i)} \Big ),
\end{split}\ee
provided we take $N$ to be the least integer satisfying $N \ge \frac{1}{p} \log_2 C'_p +1,$ which implies that $2^{pN -p} / C'_p \ge 1,$ as $h$ is decreasing. This proves the first inequality in \eqref{eq:majorantSingularValueControl}, from which the second one follows.

Since $x \mapsto \tilde{\mu} (x)$ is sublinear, we have
\be\begin{split}
\tau \Big [ \Phi \Big ( \sum_{i \in \mathbb{Z}} 2^i a_i \Big ) \Big ]
\le & \int^{\8}_0 \Phi \Big ( \sum_{i \in \mathbb{Z}} 2^i \tilde{\mu}_t ( a_i ) \Big ) d t
\lesssim \int^{\8}_0 \Phi \Big [ \sum_{i \in \mathbb{Z}} 2^i h \Big ( \frac{t}{\tau(e_i)} \Big ) \Big ] d t\\
= & \int^{\8}_0 \Phi \Big [ \int^{\8}_0 \sum_{i \in \mathbb{Z}} 2^i \chi_{(0, \tau(e_i)]} \Big ( \frac{t}{s} \Big ) (- h' ( s )) d s \Big ] d t.\\
\end{split}\ee
Note that $\mu_t (\tilde{x}) = \sum_{i \in \mathbb{Z}} 2^i \chi_{(0, \tau(e_i)]} (t),$ we have
\be\begin{split}
\tau \Big [ \Phi \Big ( \sum_{i \in \mathbb{Z}} 2^i a_i \Big ) \Big ] & \lesssim \int^{\8}_0 \Phi \Big [ \int^{\8}_0  \mu_{\frac{t}{s}} (\tilde{x})  (- h' ( s )) d s \Big ] d t  \le \int^{\8}_0 \Phi \Big [ \int^{\8}_0 \tilde{\mu}_{\frac{t}{s}} (\tilde{x}) (- h' ( s )) d s \Big ]  d t.\\
\end{split}\ee
Define $T: L_1 (\M) + L_{\8} (\M) \mapsto L_1 (0, \8) + L_{\8} (0, \8)$ by
\be
(T x) (t) = \int^{\8}_0 \tilde{\mu}_{\frac{t}{s}} (x) (- h' ( s )) d s,\quad \forall t >0.
\ee
Then
\be
\| T x \|_\beta \le \int^{\8}_0 \| \tilde{\mu}_{\frac{\cdot}{s}} (x) \|_\beta (- h' ( s )) d s = C_{p,\beta} \| \tilde{\mu} (x) \|_\beta \le C_{p,\beta} \| x \|_{L_\beta (\M)},
\ee
where the last inequality is obtained by the classical Hardy-Littlewood inequality: the mapping $f \mapsto \frac{1}{t}\int_{0}^{t}|f(s)|ds$ is bounded in $L_\beta (0,\infty)$ provided $1 < \beta \le \8.$ Also, it is easy to check that $T$ is of type $(\8, \8).$ Thus, by Theorem 2.1 in \cite{BC2012} we conclude that
\be\begin{split}
\tau \Big [ \Phi \Big ( \sum_{i \in \mathbb{Z}} 2^i a_i \Big ) \Big ] \lesssim \int^{\8}_0 \Phi \Big [ \int^{\8}_0 \tilde{\mu}_{\frac{t}{s}} (\tilde{x}) (- h' ( s )) d s \Big ]  d t \lesssim \tau \big [ \Phi (\tilde{x} ) \big ].
\end{split}\ee
Since $\tilde{x} \le 2 x,$ $S$ is order preserving, and so
\be
S_n ( x) \le S_n (\tilde{x}) \le \sum_{i \in \mathbb{Z}} 2^i a_i,\quad \forall n \ge 1,
\ee
we conclude \eqref{eq:PhiMaxOper}.
\end{proof}

\begin{remark}\label{re:Inter}
\begin{enumerate}[\rm (1)]\rm

\item The classical Marcinkiewicz interpolation theorem has been extended to include Orlicz spaces as interpolation classes by A. Zygmund, A. P. Calder\'{o}n {\it et al.} (cf. e.g. \cite{M1989}). The noncommutative analogue of this associated with a convex function was recently obtained in \cite{BC2012}. Theorem \ref{th:Inter} can be considered as a noncommutative Marcinkiewicz type interpolation theorem for a convex function of maximal operators.

\item Should Theorem \ref{th:Inter} be true whenever $S$ is simultaneously of weak type $(p_0, p_0)$ and of type $(p_1, p_1)$ and $\Phi$ an Orlicz function such that $1 \le p_0 <p_{\Phi}\le q_{\Phi}< p_1 \le \8$ (i.e., the case $p_1<\8$ is included). This was raised as an open question in the preliminary version of this paper, and affirmatively answered by Dirksen \cite{Dirk2012b} recently through extending the argument presented here to the case of restricted weak type inequalities for noncommutative maximal operators. However, Theorem \ref{th:Inter} is sufficient for our purpose of proving noncommutative maximal inequalities associated with convex functions (see Theorems \ref{th:PhiDoob} and \ref{th:NcMaxErgodi} below).

\end{enumerate}
\end{remark}

The argument presented above can be used to obtain the corresponding interpolation theorem for maximal operators on noncommutative symmetric spaces, which is different from that of \cite{Dirk2012a} and slightly simpler. Let $E$ be a rearrangement invariant (r.i., in short) Banach space with the Boyd indices $p_E \le q_E$ (for details on r.i. spaces we refer to \cite{LT1979}). Further, let $E(\M, \tau)$ be the associated noncommutative symmetric space (see e.g. \cite{DDP1989, DDP1992, KS2008, Xu1991}).

\begin{theorem}\label{th:InterSymspace}
Let $S = (S_n)_{n \ge 0}$ be a sequence of maps from $L^+_1(\mathcal{M}) + L^+_{\8}(\mathcal{M}) \mapsto L^+_{0}(\mathcal{M}).$ Assume that $S$ is subadditive and order preserving. Let $1 \le p < \8.$ Let $E$ be a rearrangement invariant space such that $p_E > p.$ If $S$ is simultaneously of weak type $(p, p)$ with constant $C_p$ and of type $(\8, \8)$ with constant $C_{\8},$ then there exists a positive constant $C_E$ depending only on $C_p, C_{\8}, p$ and $p_E,$ such that for any $x \in E^+ (\M, \tau)$ there exists $a \in E^+ (\M, \tau)$ satisfying
\begin{equation}\label{eq:InterSymspace}
\| a \|_{E (\M, \tau)} \le C_E \| x \|_{E(\M, \tau)} \quad \text{and}\quad S_n (x) \le a, \quad \forall n \ge 0.
\end{equation}
\end{theorem}

\begin{proof}
Indeed, the construction of the majorant of $S = (S_n (\tilde{x}))$ in the proof for Theorem \ref{th:Inter} is clearly valid. Hence, we have
\be\begin{split}
\Big \|  \sum_{i \in \mathbb{Z}} 2^i a_i \Big \|_{E(\M, \tau)} & \lesssim \Big \| \int^{\8}_0  \mu_{\frac{\cdot}{s}} (\tilde{x})  (- h' ( s )) d s \Big \|_E  \le \int^{\8}_0  \|D_s \mu (\tilde{x})\|_E  (- h' ( s )) d s\\
& \le \int^{\8}_0  \|D_s\|_E  (- h' ( s )) d s \| \tilde{x} \|_{E(\M, \tau)}.
\end{split}
\ee
Here $D_s$ ($0<s <\8$) are linear operators acting on measurable functions $f$ on $(0, \8)$ defined by
\be
(D_s f) (t) = f (t/s),\quad 0 < t < \8.
\ee
It is known (cf. \cite[Sect. 2.b]{LT1979}) that for $1< \beta < p_E$ there is a constant $C_{E, \beta}>0$ such that
\be
\| D_s \|_E \le C_{E, \beta} s^{\frac{1}{\beta}},\quad \forall 1 < s < \8.
\ee
Thus
\be
\int^{\8}_0  \|D_s\|_E  (- h' ( s )) d s \le C_{E, \beta} \int^{\8}_1 s^{\frac{1}{\beta} - \delta -1} d s =C_{p_E, \beta, p} < \8.
\ee
This completes the proof.
\end{proof}

With the help of Theorem \ref{th:InterSymspace}, the associated maximal inequalities on noncommutative symmetric spaces are in order, including Doob's inequality, Dunford-Schwartz and Stein maximal ergodic inequalities, as well as the corresponding pointwise convergence theorems (see \cite{JX2007} for detailed information). We omit the details.

\section{Main results}\label{PrTh}

Let $\Phi$ be an Orlicz function. As noted in \cite[Remark 1.1]{BC2012}, if $1 < p_{\Phi} \le q_{\Phi}< \8,$ then for any noncommutative $L_{\Phi}$-martingale $x = (x_n),$ there exists a unique $x_{\8} \in L_{\Phi} (\M)$ such that $x_n = \mathcal{E}_n (x_{\8})$ for all $n.$ We simply write $x_{\8} = x$ in this case.

Our first main result is the following noncommutative Doob inequality associated with a convex function, generalizing Junge's noncommutative Doob inequality in $L_p$ \cite{Junge2002}.

\begin{theorem}\label{th:PhiDoob}
Let $\M$ be a finite von Neumann algebra with a normalized normal faithful trace $\tau,$ equipped with a filtration $( \M_n )_{n\ge 0}$ of von Neumann subalgebras of $\M.$ Let $\Phi$ be an Orlicz function and $x = ( x_n )$ be a noncommutative $L_{\Phi}$-martingale with respect to $({\mathcal{M}}_{n}).$ If $1 < p_{\Phi} \le q_{\Phi}< \8,$ then
\beq\label{eq:PhiDoob}
\tau \Big [ \Phi \big ( {\sup_n}^+ x_n \big ) \Big ] \thickapprox \tau \big [ \Phi( |x| ) \big ].
\eeq
\end{theorem}

\begin{proof}

Decomposing an operator into a linear combination of four positive ones, by Proposition \ref{prop:BasicPhiMax} (2) we can assume that $x = (x_n)$ is a positive martingale in $L_{\Phi} (\M).$ Let $S = (\mathcal{E}_n).$ By Cuculescu's weak type $(1,1)$ maximal martingale inequality \cite{Cucu1971}, we see that $S$ is of weak type $(1,1).$ Also, $S$ is trivially of type $(\8, \8),$ due to the well known fact that conditional expectations are contractions for the operator norm. Thus, by Theorem \ref{th:Inter} we conclude that
\be
\tau \Big [ \Phi \big ( {\sup_n}^+ x_n \big ) \Big ] \lesssim \tau \big [ \Phi( |x| ) \big ].
\ee

To prove the converse inequality, consider a decomposition $x_n = a y_n b$ for all $n$ and $\sup_n \| y_n\|_{\8} \le 1.$ One has
\be\begin{split}
\tau \big [ \Phi( |x| ) \big ] & = \int^{\8}_0 \Phi (\mu_t (x)) d t\\
& \le 2 \sup_n \| y_n \|_{\8} \int^{\8}_0 \Phi [ \mu_t ( |a|) \mu_t ( |b|) ] d t\\
& \le 2 \sup_n \| y_n \|_{\8} \int^{\8}_0 \Phi \Big [ \frac{1}{2} \Big ( \mu_t ( |a|)^2 +  \mu_t ( |b|)^2 \Big ) \Big ] d t\\
& \le 2 \sup_n \| y_n \|_{\8} \frac{1}{2} \Big ( \tau \big [ \Phi ( |a|^2) \big ] +  \tau \big [ \Phi ( |b|^2) \big ] \Big ).
\end{split}\ee
Thus,
\be
\tau \big [ \Phi( |x| ) \big ] \le 2\tau \Big [ \Phi \big ( {\sup_n}^+ x_n \big ) \Big ].
\ee
This completes the proof.
\end{proof}

\begin{remark}\label{re:PhiMaxSpace}\rm
Let $\Phi$ be an Orlicz function. We define the Hardy-Orlicz maximal space of noncommutative martingales as
\be
\mathcal{H}^{\mathrm{max}}_{\Phi} (\M): = \Big \{ x \in L_{\Phi} (\M):\; \| x \|_{\mathcal{H}^{\mathrm{max}}_{\Phi}} = \big \| {\sup_n}^+ \mathcal{E}_n (x) \big \|_{\Phi} < \8 \Big \}.
\ee
(See \cite[Sect. 4]{JX2005} for the case $\Phi (t) = t^p$.) Then, Theorem \ref{th:PhiDoob} implies that $\mathcal{H}^{\mathrm{max}}_{\Phi} (\M) = L_{\Phi} (\M)$ with equivalent norms provided $1 < p_{\Phi} \le q_{\Phi}<\8.$
\end{remark}

As a consequence of Theorem \ref{th:PhiDoob}, we obtain the following noncommutative Burkholder-Davis-Gundy inequality associated with a convex function.

\begin{corollary}\label{cor:NcBDG}
Let $\M$ be a finite von Neumann algebra with a normalized normal faithful trace $\tau,$ equipped with a filtration $( \M_n )_{n\ge 0}$ of von Neumann subalgebras of $\M.$ Let $\Phi$ be an Orlicz function, and let $x=( x_{n} )_{n\geq 0}$ be a noncommutative $L_{\Phi}$-martingale with respect to $( \M_n )_{n \ge 0}.$ If $1<p_{\Phi} \le q_{\Phi}<2,$ then
\beq\label{eq:NcPhiBDGle2}
\begin{split}
\tau \Big ( \Phi \Big [ {\sup_n}^+ x_n \Big ] \Big ) \approx \inf \bigg \{ \tau \Big ( \Phi \Big [ \Big ( \sum_{k= 0 }^{\infty} |dy_{k}|^{2}
\Big )^{ \frac{1}{2}} \Big ] \Big ) + \tau \Big ( \Phi \Big [ \Big ( \sum_{k= 0 }^{\infty} |dz_{k}^{*}|^{2} \Big )^{ \frac{1}{2}} \Big ] \Big ) \bigg \},
\end{split}
\eeq
where the infimum runs over all decomposition $x_n = y_n + z_n$ with $y_n$ in $\mathcal{H}^C_{\Phi}({\mathcal{M}})$ and $z_n$ in $\mathcal{H}^R_{\Phi}({\mathcal{M}});$ and if $2 <p_{\Phi} \le q_{\Phi}<\infty,$ then
\beq\label{eq:NcPhiBDGge2}
\begin{split}
\tau \Big ( \Phi  \Big [ {\sup_n}^+ x_n \Big ] \Big ) \approx \tau \Big ( \Phi \Big [ \Big ( \sum_{k= 0 }^{\infty} |dx_{k}|^{2} \Big )^{ \frac{1}{2}} \Big ] \Big ) + \tau \Big ( \Phi \Big [ \Big ( \sum_{k= 0 }^{\infty} |dx_{k}^{*}|^{2} \Big )^{ \frac{1}{2}} \Big ] \Big ).
\end{split}\eeq
\end{corollary}

\begin{remark}\label{rk:NcBDG}\rm
The classical case of Corollary \ref{cor:NcBDG} was originally proved by Burkholder, Davis, and Gundy in \cite{BDG1972} (see also \cite{Burk1973}). Note that, the classical case holds even if $p_\Phi =1$ (e.g. \cite{Davis1970}). However, the noncommutative case is surprisingly different. Indeed, it was shown in \cite[Corollary 14]{JX2005} that $\mathcal{H}_1 \not= \mathcal{H}^{\mathrm{max}}_1.$ This implies that \eqref{eq:NcPhiBDGle2} does not hold when $ \Phi (t) = t$ for which $p_\Phi =1.$
\end{remark}

\begin{proof}\; It is proved in \cite{BC2012} that if $1<p_{\Phi} \le q_{\Phi}<2,$ then
\beq\label{eq:NcPhiBGle2}
\begin{split}
\tau \left [ \Phi( |x| ) \right ] \thickapprox \inf \bigg \{ \tau \Big ( \Phi \Big [ \Big ( \sum_{k= 0 }^{\infty} |dy_{k}|^{2} \Big )^{ \frac{1}{2}} \Big ] \Big ) + \tau \Big ( \Phi \Big [ \Big ( \sum_{k= 0 }^{\infty} |dz_{k}^{*}|^{2} \Big )^{ \frac{1}{2}} \Big ] \Big ) \bigg \},
\end{split}
\eeq
where the infimum runs over all decomposition $x_n = y_n + z_n$ with $y_n$ in $\mathcal{H}^C_{\Phi}({\mathcal{M}})$ and $z_n$ in $\mathcal{H}^R_{\Phi}({\mathcal{M}});$ and if $2 <p_{\Phi} \le q_{\Phi}<\infty,$ then
\beq\label{eq:NcPhiBGge2}
\tau \left [ \Phi( |x| ) \right ] \thickapprox \tau \Big ( \Phi \Big [ \Big ( \sum_{k= 0 }^{\infty} |dx_{k}|^{2} \Big )^{ \frac{1}{2}} \Big ] \Big ) + \tau \Big ( \Phi \Big [ \Big ( \sum_{k= 0 }^{\infty} |dx_{k}^{*}|^{2} \Big )^{ \frac{1}{2}} \Big ] \Big ).
\eeq
An appeal to \eqref{eq:PhiDoob} yields the required inequalities \eqref{eq:NcPhiBDGle2} and \eqref{eq:NcPhiBDGge2}.
\end{proof}

\begin{remark}\label{re:NcPhiBG}\rm
We remark that there is a gap in the proof of \eqref{eq:NcPhiBGge2} in \cite{BC2012}, as pointed out to the first two named authors by Q. Xu. This was recently resolved by Dirksen and Ricard \cite{DR}.
\end{remark}

Now we turn to noncommutative maximal ergodic inequalities associated with convex functions. To state our results, we need some more notation.

Let $\M$ be a semifinite von Neumann algebra with a normal semifinite faithful trace $\tau,$ and let $L_p(\M)$ be the associated noncommutative $L_p$-space. Consider a linear map $T:\; \M \mapsto \M$ which may satisfy the
following conditions:
\begin{enumerate}[{\rm (I)}]

\item $T$ is a contraction on $\M,$ that is, $\| T x \| \le \|x\|$ for all $x \in \M.$

\item $T$ is positive, i.e., $Tx \ge 0$ if $x \ge 0.$

\item $\tau \circ T \le \tau,$ that is, $\tau (T x) \le \tau (x)$ for all $x \in L_1 (\M) \cap \M_+.$

\item $T$ is symmetric relative to $\tau,$ i.e., $\tau ( (Ty)^* x ) = \tau (y^* T x )$ for all $x, y \in L_2 (\M) \cap \M.$

\end{enumerate}
Under conditions (I)-(III), $T$ naturally extends to a contraction on $L_p ( \M )$ for every $1 \le p < \8.$
The extension will be still denoted by $T.$


\begin{theorem}\label{th:NcMaxErgodi}
Let $\Phi$ be an Orlicz function with $1 < p_{\Phi} \le q_{\Phi} < \8.$ If $T: \M \mapsto \M$ is a linear map satisfying $(I) - (III),$ then
\beq\label{eq:NcPhiDS}
\tau \Big ( \Phi \big [ {\sup_n}^+ M_n (x) \big ] \Big ) \lesssim \tau ( \Phi [ |x| ] ),\quad \forall x \in L_{\Phi} (\M),
\eeq
where $M_n : = \frac{1}{n+1} \sum^n_{k=0} T^k$ for any $n \ge 1.$ If, in addition, $T$ satisfies (IV), then
\beq\label{eq:NcPhiStein}
\tau \Big ( \Phi \big [ {\sup_n}^+ T^n (x) \big ] \Big ) \lesssim \tau ( \Phi [ |x| ] ),\quad \forall x \in L_{\Phi} (\M).
\eeq
\end{theorem}

The inequalities \eqref{eq:NcPhiDS} and \eqref{eq:NcPhiStein} are the noncommutative forms of the classical Dunford-Schwartz and Stein maximal ergodic inequality for a convex function of positive and symmetric positive contractions. These generalize the noncommutative Dunford-Schwartz and Stein maximal ergodic inequalities of Junge and Xu \cite{JX2007} in the $L^p$ case to the case of convex functions. The proofs of \eqref{eq:NcPhiDS} and \eqref{eq:NcPhiStein} are again based on Theorem \ref{th:Inter}.

\begin{proof}\; Decomposing an operator into a linear combination of four positive ones, we can assume $x \in L^+_{\Phi} (\M).$ Let $S = (M_n).$ Each $M_n$ is considered to be a map on $L^+_1 (\M) + L^+_{\8} (\M),$ positive and additive (and so subadditive too). Yeadon's weak type $(1,1)$ maximal ergodic inequality in \cite{Yeadon1977} says that $S$ is of weak type $(1,1).$ Also, $S$ is evidently of type $(\8, \8).$ Then, we deduce \eqref{eq:NcPhiDS} from Theorem \ref{th:Inter}.

On the other hand, let $S = (T^n).$ Then $S$ is additive and so subadditive. By \cite[Theorem 5.1]{JX2007}, $S$ is of type $(p, p)$ for every $1 < p \le \8.$ An appeal to Theorem \ref{th:Inter} immediately yields \eqref{eq:NcPhiStein}.
\end{proof}

\

Let us present two examples illustrating applications of the inequalities obtained above.

\begin{example}\label{ex:pleq}\rm
Let $\Phi (t) = t^a \ln (1 + t^b)$ with $a > 1$ and $b >0.$ It is easy to check that $\Phi$ is an Orlicz function and
\be
p_{\Phi} = a\quad \text{and}\quad q_{\Phi} = a + b.
\ee
Thus, both Theorems \ref{th:PhiDoob} and \ref{th:NcMaxErgodi} can be applied to this function. Furthermore, if $1< a < a+b<2,$ then \eqref{eq:NcPhiBDGle2} holds true; if $a > 2,$ then \eqref{eq:NcPhiBDGge2} is valid. Unfortunately, when $1<a \le 2 \le a + b,$ then Corollary \ref{cor:NcBDG} gives no information.
\end{example}

\begin{example}\label{ex:p=q=2}\rm
Let $\Phi (t) = t^p (1 + c \sin(p \ln t))$ with $p > 1/(1-2c)$ and $0< c <1/2.$ Then $\Phi$ is an Orlicz function and
\be
p_{\Phi} = q_{\Phi} = p.
\ee
Therefore, Theorems \ref{th:PhiDoob} and \ref{th:NcMaxErgodi} can be applied to this function, and so does Corollary \ref{cor:NcBDG} except the case $p=2.$
\end{example}

\section{Weak type maximal inequalities}\label{Ex}


All the results continue to hold if we replace the noncommutative maximal operator $\tau [ \Phi ({\sup_n}^+ x_n) ]$ associated with a convex function by a certain weak maximal operator. The required modifications are not difficult and left to the interested reader. However, for
the sake of convenience, we write the corresponding definitions and results, and some main points of Theorem \ref{th:InterW}. We refer to \cite{BCLJ2011} for noncommutative weak Orlicz spaces and for the terminology used here.

Let $\Phi$ be an Orlicz function. For $x \in L^w_{\Phi} (\M),$ we set
\be
\| x \|_{\Phi, \8} = \sup_{t >0} t \Phi \big [ \mu_t (x) \big ].
\ee
When $\Phi (t) = t^p,$ $\| x \|_{\Phi, \8}$ is just the usual weak $L_p$-norm $\| x \|_{p, \8}.$

The following is the definition of a weak type maximal operator associated with a convex function.

\begin{definition}\label{df:PhiMaxOperW}
Let $(x_n)$ be a sequence in $L^w_{\Phi} (\M).$ We define $\| {\sup_n}^+ x_n \|_{\Phi, \8}$ by
\beq\label{eq:PhiMaxSequ}
\big \| {\sup_n}^+ x_n \big \|_{\Phi, \8} : = \inf \left \{ \frac{1}{2} \Big ( \big \| |a|^2 \big \|_{\Phi, \8} +  \big \| |b|^2 \big \|_{\Phi, \8} \Big ) \sup_n \| y_n \|_{\8} \right \}
\eeq
where the infimum is taken over all decompositions $x_n = a y_n b$ for $a, b \in L_0 (\M)$ and $(y_n) \subset L_{\8} (\M)$ with $|a|^2, |b|^2 \in L^w_{\Phi} (\M),$ and $\| y_n \|_{\8} \le 1$ for all $n.$
\end{definition}

We have the noncommutative Marcinkiewicz type interpolation theorem for the weak type maximal operator associated with a convex function as follows. To this end, recall that
\be
a_{\Phi} = \inf_{t>0} \frac{t \Phi'(t)}{\Phi (t)}\quad \text{and} \quad b_{\Phi} = \sup_{t>0} \frac{t \Phi'(t)}{\Phi (t)}.
\ee
Note that $1 \le a_\Phi \le p_\Phi \le q_\Phi \le b_\Phi,$ but they do not coincide in general (see \cite{M1985, M1989} for details).

\begin{theorem}\label{th:InterW}
Suppose $1 \le p_{0}<p_{1} \leq \infty.$ Let $S = (S_n)_{n \ge 0}$ be a sequence of maps from $L^+_{p_0}(\mathcal{M}) + L^+_{p_1}(\mathcal{M}) \mapsto L^+_{0}(\mathcal{M}).$ Assume that $S$ is subadditive. If $S$ is of weak type $(p_0, p_0)$ with constant $C_0$ and of type $(p_1, p_1)$ with constant $C_1,$ then for an Orlicz function $\Phi$ with $p_{0}<a_{\Phi}\le b_{\Phi}<p_{1},$ there exists a positive constant $C$ depending only on $p_0, \; p_1, C_0, C_1$ and $\Phi,$ such that
\begin{equation}\label{eq:PhiMaxOperW}
\big \| {\sup_n}^+ S_n( x) \big \|_{\Phi, \8} \leq C \left \| x \right \|_{\Phi, \8},
\end{equation}
for all $x \in L^w_{\Phi}(\mathcal{M})_+.$
\end{theorem}

\begin{proof}
We give the main point of the proof. Indeed, modifying slightly the proof of \cite[Theorem 3.1]{JX2007} we conclude that for $p_0< p_0' < a_{\Phi} \le b_{\Phi}< p_1' < p_1 \le \8,$
\be
\big \| {\sup_n}^+ S_n( x) \big \|_{p'_i, \8} \leq C_{p'_i} \left \| x \right \|_{p'_i, \8},\quad i =0,1,
\ee
that is, for each $x_i \in L^+_{p'_i} (\M)$ there exists $a_i \in L^+_{p'_i} (\M)$ such that
\beq\label{eq:p-MaxOperW}
\left \| a_i \right \|_{p'_i, \8} \leq C \left \| x_i \right \|_{p'_i, \8}\quad \text{and} \quad S_n (x_i) \le a_i,\quad \forall n \ge 1.
\eeq
(This can be also obtained by Theorem \ref{th:InterSymspace} above.)

Now, take $x \in L^w_{\Phi} (\M)_+.$ For any $\alpha > 0$ let $x = x^{\alpha}_0 + x^{\alpha}_1,$ where $x^{\alpha}_0 = x e_{(\alpha, \8)}(x).$ By \eqref{eq:p-MaxOperW}, for $x^{\alpha}_i$ there exists a corresponding $a_i$ ($i=0,1$). The remainder of the proof is the same as that of \cite[Theorem 4.2]{BCLJ2011}.
\end{proof}

The following is a noncommutative Doob weak type inequality associated with a convex function.

\begin{theorem}\label{th:PhiDoobW}
Let $\M$ be a finite von Neumann algebra with a normalized normal faithful trace $\tau,$ equipped with a filtration $({\mathcal{M}}_{n})$ of von Neumann subalgebras of ${\mathcal{M}}.$ Let $\Phi$ be an Orlicz function and let $x = ( x_n )$ be a noncommutative $L^w_{\Phi}$-martingale with respect to $({\mathcal{M}}_{n}).$ If $1 < a_{\Phi} \le b_{\Phi}< \8,$ then
\beq\label{eq:PhiDoobW}
\big \| {\sup_n}^+ x_n \big \|_{\Phi, \8} \thickapprox \| x \|_{\Phi, \8}.
\eeq
\end{theorem}

Combining this with \cite[Theorem 5.8]{BCLJ2011} and the associated result in \cite{DR} yields the noncommutative Burkholder-Davis-Gundy weak type inequality associated with a convex function as follows.

\begin{corollary}\label{th:NcBDGW}
Let $\M$ be a finite von Neumann algebra with a normalized normal faithful trace $\tau,$ equipped with a filtration $({\mathcal{M}}_{n})$ of von Neumann subalgebras of $\M.$ Let $\Phi$ be an Orlicz function and let $x=( x_{n} )_{n\geq 0}$ be a noncommutative $L_{\Phi}$-martingale with respect to $( \M_n )_{n \ge 0}.$ If $1<a_{\Phi} \le b_{\Phi}<2,$ then
\beq\label{eq:NcPhiBDGWle2}
\begin{split}
\big \| {\sup_n}^+ x_n \big \|_{\Phi, \8} \approx \inf \left \{ \bigg \| \bigg ( \sum_{k= 0 }^{\infty} |dy_{k}|^{2} \bigg )^{\frac{1}{2}}
\bigg \|_{\Phi, \8} + \bigg \| \bigg ( \sum_{k= 0 }^{\infty} |dz_{k}^{*}|^{2} \bigg )^{\frac{1}{2}} \bigg \|_{\Phi, \8} \right \},
\end{split}
\eeq
where the infimum runs over all decomposition $x_n = y_n + z_n$ with $(y_n)$ in $L^w_{\Phi}(\mathcal{M}; \el^2_C)$ and $(z_n)$ in $L^w_{\Phi}(\mathcal{M}; \el^2_R);$ and if $2 <a_{\Phi} \le b_{\Phi}<\infty,$ then
\beq\label{eq:NcPhiBDGWge2}
\big \| {\sup_n}^+ x_n \big \|_{\Phi, \8} \approx \bigg \| \bigg ( \sum_{k= 0 }^{\infty} |dx_{k}|^{2} \bigg )^{\frac{1}{2}} \bigg \|_{\Phi, \8} + \bigg \| \bigg ( \sum_{k= 0 }^{\infty} |dx_{k}^{*}|^{2} \bigg )^{\frac{1}{2}} \bigg \|_{\Phi, \8}.
\eeq
\end{corollary}

The weak type analogue of Theorem \ref{th:NcMaxErgodi} concerning maximal ergodic inequalities associated with a convex function is similar and omitted.

\subsection*{Acknowledgement} We are grateful to Prof. Xu for useful suggestions on this paper. We also thank the anonymous referee for making helpful comments and suggestions, which have been incorporated into this version of the paper. T.N. Bekjan is partially supported by NSFC grant No.11371304. Z. Chen is partially supported by NSFC grant No.11171338 and No.11431011. A. Os\c{e}kowski is supported in part by MNiSW Grant N N201 364436.


\begin{thebibliography}{99}

\bibitem{AAP1982} C. A. Akemann, J. Anderson, G. K. Pedersen,
Triangle inequalities in operator algebras,
{\it Linear Multilinear Algebra} {\bf 11} (1982), 167-178.



\bibitem{Bek2008} T. N. Bekjan,
Noncommutative maximal ergodic theorems for positive contractions,
{\it J. Funct. Anal.} {\bf 254} (2008), 2401-2418.

\bibitem{BC2012} T. N. Bekjan, Z. Chen,
Interpolation and $\Phi$-moment inequalities of noncommutative martingales,
{\it Probab. Theory Relat. Fields} {\bf 152} (2012), 179-206.

\bibitem{BCLJ2011} T. N. Bekjan, Z. Chen, P. Liu, Y. Jiao,
Noncommutative weak Orlicz spaces and martingale inequalities,
{\it Studia Math.} {\bf 204}(3) (2011), 195-212.

\bibitem{BCPY2010} T. N. Bekjan, Z. Chen, M. Perrin, Z. Yin,
Atomic decomposition and interpolation for Hardy spaces of noncommutative martingales,
{\it J. Funct. Anal.} {\bf 258} (2010), 2483-2505.





\bibitem{Burk1973} D. L. Burkholder,
Distribution function inequalities for martingales,
{\it Ann. Probab.} {\bf 1}(1) (1973), 19-42.


\bibitem{BDG1972} D. L. Burkholder, B. Davis, R. Gundy,
Integral inequalities for convex functions of operators on martingales,
{\it Proc. 6th Berkley Symp.} {\bf II} (1972), 223-240.


\bibitem{CXY2013} Z. Chen, Q. Xu, Z. Yin,
Harmonic analysis on quantum tori,
{\it Commun. Math. Phys.} {\bf 322} (2013), 755-805.

\bibitem{Cucu1971} I. Cuculescu,
Martingales on von Neumann algebras,
{\it J. Multivariate. Anal.} {\bf 1} (1971), 17-27.

\bibitem{Davis1970} B. J. Davis,
On the integrability of the martingale square function,
{\it Israel J. Math.} {\bf 8} (1970), 187-190.

\bibitem{DJ2004} A. Defant, M. Junge,
Maximal theorems of Menchoff-Rademacher type in non-commtative $L_q$-spaces,
{\it J. Funct. Anal.} {\bf 206} (2004), 322-355.

\bibitem{Dirk2012a} S. Dirksen,
Noncommutative Boyd interpolation theorems,
to appear in {\it Trans. Amer. Math. Soc.}

\bibitem{Dirk2012b} S. Dirksen,
Weak-type interpolation for noncommutative maximal operators,
arXiv:1212.5168.

\bibitem{DR} S. Dirksen, E. Ricard,
Some remarks on noncommutative Khintchine inequalities,
{\it Bull. London Math. Soc.} {\bf 45} (2013), 618-624.

\bibitem {DDP1989} P. G. Dodds, T. K. Dodds, B. de Pagter,
Noncommutative Banach function spaces,
{\it Math. Z. } {\bf 201} (1989), 583-587.

\bibitem {DDP1992} P. G. Dodds, T. K. Dodds, B. de Pagter,
Fully symmetric operator spaces,
{\it Integ. Equ. Oper. Theory} {\bf 15} (1992), 942-972.

\bibitem{FK1986} T. Fack, H. Kosaki,
Generalized $s$-numbers of $ \tau $-measure operators,
{\it Pacific J. Math. } {\bf 123} (1986), 269-300.

\bibitem{ER2000} E. Effros, Z. J. Ruan,
{\it Operator Spaces,}
Oxford University Press, Oxford, 2000.



\bibitem{HP2000} F. Hiai, D. Petz,
{\it The Semicircle Law, Free Random Variables and Entropy,}
Amer. Math. Soc., Providence, RI, 2000.



\bibitem{Junge2002} M. Junge,
Doob's inequality for non-commutative martingales,
{\it J. Reine Angew. Math.} {\bf 549} (2002), 149-190.

\bibitem{Junge2005} M. Junge,
Embedding of the operator space OH and the logarithmic `little Grothendieck inequality',
{\it Invent. Math.} {\bf 161} (2005), 225-286.

\bibitem{JLeMX2006} M. Junge, C. Le Merdy, Q. Xu,
$\mathrm{H}^{\8}$ functional calculus and square functions on noncommutative $L_p$-spaces,
{\it Ast\'{e}risque} {\bf 305} (2006), vi + 138.


\bibitem{JP2008} M. Junge, J. Parcet,
Operator space embedding of Schatten $p$-classes into von Neumann algebra preduals,
{\it Geom. Funct. Anal.} {\bf 18} (2008), 522-551.

\bibitem{JX2003} M. Junge, Q. Xu,
Noncommutative Burkholder/Rosenthal inequalities,
{\it Ann. Probab.} {\bf 31} (2003), 948-995.

\bibitem{JX2005} M. Junge, Q. Xu,
On the best constants in some noncommutative martingale inequalities,
{\it Bull. London Math. Soc.} {\bf 37} (2005), 243-253.

\bibitem{JX2007} M. Junge, Q. Xu,
Noncommutative maximal ergodic inequalities,
{\it J. Amer. Math. Soc.} {\bf 20}(2) (2007), 385-439.

\bibitem{JX2008} M. Junge, Q. Xu,
Noncommutative Burkholder/Rosenthal inequalities II: Applications,
{\it Israel J. Math.} {\bf 167} (2008), 227-282.

\bibitem{JX2010} M. Junge, Q. Xu,
Representation of certain homogeneous Hilbertian operator spaces and applications,
{\it Invent. Math.} {\bf 175} (2010), 75-118.

\bibitem{KS2008} N. J. Kalton, F. A. Sukochev,
Symmetric norms and spaces of operators,
{\it J. Reine Angew. Math.} {\bf 621} (2008), 81-121.


\bibitem{LMX2010} C. Le Merdy, Q. Xu,
Maximal theorems and square functions for analytic operators on $L^p$-spaces,
{\it J. London Math. Soc.} {\bf 86}(2) (2012), 343-365.

\bibitem{LMX2011} C. Le Merdy, Q. Xu,
Strong $q$-variation inequalities for analytic semigroups,
{\it Ann .Inst. Fourier.} {\bf 62} (2012), 2069-2097.

\bibitem{LT1979} J. Lindenstrauss, L. Tzafriri,
{\it Classical Banach Space \rm{II},}
Springer-Verlag, Berlin, 1979.





\bibitem{M1985} L. Maligranda,
Indices and interpolation,
{\it Dissert. Math.} {\bf 234}, Polska Akademia Nauk, Inst. Mat., 1985.

\bibitem{M1989} L. Maligranda,
{\it Orlicz spaces and interpolation,}
Seminars in Mathematics, Departamento de Matem\'{a}tica, Universidade Estadual de Campinas, Brasil, 1989.






\bibitem{PR2006} J. Parcet, N. Randrianantoanina,
Gundy's decomposition for noncommutative martingales and applications,
{\it Proc. London Math. Soc.} {\bf 93}(3) (2006), 227-252.

\bibitem{Perrin2009} M. Perrin,
A noncommutative Davis' decomposition for martingales,
{\it J. London Math. Soc.} (2){\bf 80}(3) (2009), 627-648.


\bibitem{Pisier1998} G. Pisier,
Non-commutative vector valued $L_p$-spaces and completely $p$-summing maps,
{\it Ast\'{e}risque} {\bf 247} (1998), v + 131.

\bibitem{Pisier2003} G. Pisier,
{\it Introduction to Operator Space Theory,}
Cambridge University Press, Cambridge, 2003.

\bibitem{PS2002} G. Pisier, D. Shlyakhtenko,
Grothendieck's theorem for operator spaces,
{\it Invent. Math.} {\bf 150} (2002), 185-217.

\bibitem {PX1997} G. Pisier, Q. Xu,
Non-commutative martingale inequalities,
{\it Commun. Math. Phys.} {\bf 189} (1997), 667-698.

\bibitem{PX2003} G. Pisier, Q. Xu,
Noncommutative $L^p$-spaces,
{\it Handbook of the Geometry of Banach Spaces, vol.2}: 1459-1517, 2003.

\bibitem{Rand2002} N. Randrianantoanina,
Non-commutative martingale transforms,
{\it J. Funct. Anal.} {\bf 194} (2002), 181-212.



\bibitem{Rand2007} N. Randrianantoanina,
Conditional square functions for noncommutative martingales,
{\it Ann. Proba.} {\bf 35} (2007), 1039-1070.

\bibitem{RX2014} E. Ricard, Q. Xu,
A noncommutative martingale convexity inequality,
arXiv: 1405.0431.


\bibitem{VDN1992} D. V. Voiculescu, K. J. Dykema, A. Nica,
{\it Free Random Variables,}
Amer. Math. Soc., Providence, RI, 1992.

\bibitem{Xu1991} Q. Xu,
Analytic functions with values in lattices and symmetric spaces of measurable operators,
{\it Math. Proc. Cambridge Phil. Soc.} {\bf 109} (1991), 541-563.

\bibitem{Xu2006} Q. Xu,
Operator space Grothendieck inequalities for noncommutative $L_p$-spaces,
{\it Duke Math. J.} {\bf 131} (2006), 525-574.

\bibitem{Xu2007} Q. Xu,
Noncommutative $L_p$-Spaces and Martingale Inequalities,
Book manuscript, 2007.

\bibitem{Yeadon1977} F. J. Yeadon,
Ergodic theorems for semifinite von Neumann algebras,
{\it J. London Math. Soc.} {\bf 16}(2) (1977), 326-332.

\end{thebibliography}
\end{document}